\documentclass[12pt]{amsart}
\usepackage{amscd}
\def\frk{\frak}               

\def\Phi{{\frk n}}
\def\Phi{{\frk N}}
\def\opn#1#2{\def#1{\operatorname{#2}}} 
\opn\chara{char} \opn\length{\ell} \opn\pd{pd} \opn\rk{rk}
\opn\projdim{proj\,dim} \opn\injdim{inj\,dim} \opn\rank{rank}
\opn\depth{depth} \opn\sdepth{sdepth} \opn\fdepth{fdepth}
\opn\grade{grade} \opn\height{height} \opn\embdim{emb\,dim}
\opn\codim{codim}  \opn\min{min} \opn\max{max}

\opn\Tr{Tr} \opn\bigrank{big\,rank}
\opn\superheight{superheight}\opn\lcm{lcm}
\opn\trdeg{tr\,deg}
\opn\reg{reg} \opn\lreg{lreg} \opn\ini{in} \opn\lpd{lpd}
\opn\size{size}
\opn\div{div} \opn\Div{Div} \opn\cl{cl} \opn\Cl{Cl}
\opn\Spec{Spec} \opn\Supp{Supp} \opn\supp{supp} \opn\Sing{Sing}
\opn\Ass{Ass} \opn\Min{Min}
\opn\Ann{Ann} \opn\Rad{Rad} \opn\Soc{Soc}
\opn\Im{Im} \opn\Ker{Ker} \opn\Coker{Coker} \opn\Am{Am}
\opn\Hom{Hom} \opn\Tor{Tor} \opn\Ext{Ext} \opn\End{End}
\opn\Aut{Aut} \opn\id{id}  \opn\deg{deg}

\opn\nat{nat}
\opn\pff{pf}
\opn\Pf{Pf} \opn\GL{GL} \opn\SL{SL} \opn\mod{mod} \opn\ord{ord}
\opn\Gin{Gin} \opn\Hilb{Hilb}
\opn\aff{aff} \opn\con{conv} \opn\relint{relint} \opn\st{st}
\opn\lk{lk} \opn\cn{cn} \opn\core{core} \opn\vol{vol}
\opn\link{link} \opn\star{star}
\opn\gr{gr}

\def\pot#1#2{#1[\kern-0.28ex[#2]\kern-0.28ex]}

\opn\dirlim{\underrightarrow{\lim}}
\opn\inivlim{\underleftarrow{\lim}}
\let\Dirsum=\bigoplus

\let\to=\rightarrow

\def\Implies{\ifmmode\Longrightarrow \else
        \unskip${}\Longrightarrow{}$\ignorespaces\fi}
\def\implies{\ifmmode\Rightarrow \else
        \unskip${}\Rightarrow{}$\ignorespaces\fi}
\def\iff{\ifmmode\Longleftrightarrow \else
        \unskip${}\Longleftrightarrow{}$\ignorespaces\fi}

\let\:=\colon
\newtheorem{Theorem}{Theorem}[section]
\newtheorem{Lemma}[Theorem]{Lemma}
\newtheorem{Corollary}[Theorem]{Corollary}
\newtheorem{Proposition}[Theorem]{Proposition}
\newtheorem{Remark}[Theorem]{Remark}

\newtheorem{Example}[Theorem]{Example}

\newtheorem{Definition}[Theorem]{Definition}

\let\epsilon\varepsilon
\let\phi=\varphi
\let\kappa=\varkappa
\def\qed{\ifhmode\textqed\fi
      \ifmmode\ifinner\quad\qedsymbol\else\dispqed\fi\fi}
\def\textqed{\unskip\nobreak\penalty50
       \hskip2em\hbox{}\nobreak\hfil\qedsymbol
       \parfillskip=0pt \finalhyphendemerits=0}
\def\dispqed{\rlap{\qquad\qedsymbol}}

\opn\dis{dis}
\def\pnt{{\raise0.5mm\hbox{\large\bf.}}}

\opn\Lex{Lex}

\begin{document}

\title{\bf The Stanley conjecture on intersections of four monomial prime ideals}

\author{ Dorin Popescu }

\thanks{The  support from  the CNCSIS grant PN II-542/2009 of Romanian Ministry of Education, Research and Inovation is gratefully acknowledged.}

\address{Dorin Popescu, Simion Stoilow Institute of Mathematics, Research unit 5,
University of Bucharest, P.O.Box 1-764, Bucharest 014700, Romania}
\email{dorin.popescu@imar.ro}

\maketitle
\begin{abstract} We show that the Stanley's Conjecture holds for  an intersection of four monomial prime ideals of a
polynomial algebra $S$ over a field and for an arbitrary intersection of monomial prime ideals
$(P_i)_{i\in [s]}$ of $S$ such that each $P_i$ is not contained in the sum of the other $(P_j)_{j\not = i}$.

  \vskip 0.4 true cm
 \noindent
\end{abstract}

\section*{Introduction}

 Let  $S=K[x_1,\ldots,x_n]$, $n\in {\bf N}$, be a polynomial ring over a field $K$. Let $I\subset S$ be a monomial ideal of $S$ and  $u\in I$ a monomial in $I$.
  For $Z\subset \{x_1,\ldots ,x_n\}$ let $uK[Z]$ be the linear $K$-subspace of $I$ generated by the elements $uf$, $f\in K[Z]$.  A  presentation of $I$ as a finite direct sum of such spaces ${\mathcal D}:\
I=\Dirsum_{i=1}^ru_iK[Z_i]$ is called a Stanley decomposition of $I$. Set $\sdepth
(\mathcal{D}):=\min\{|Z_i|:i=1,\ldots,r\}$ and
\[
\sdepth\ I :=\max\{\sdepth \ ({\mathcal D}):\; {\mathcal D}\; \text{is a
Stanley decomposition of}\;  I \}.
\]

 Stanley's Conjecture \cite{S} says that $\sdepth\ I\geq \depth\ I$. This would be a nice connection between a combinatorial invariant and a homological
one. The Stanley's Conjecture holds for arbitrary squarefree monomial ideals if $n\leq 5$
by \cite{P} (see especially the arXiv version), and for intersections of three monomial
prime ideals by \cite{AP}. In the non squarefree monomial ideals a useful inequality is $\sdepth I\leq \sdepth \sqrt{I}$ (see \cite[Theorem 2.1]{Is}). In this paper we study only the case of squarefree monomial ideals.

We will extend the so called "special Stanley decompositions" of \cite{AP} (see Theorem \ref{sdepth}). This tool is very important because it gives lower bounds of $\sdepth_S I$ in terms of $\sdepth$ of some ideals in less variables for which we may apply mathematical induction. More precisely, we use it to find ``good" lower bounds of $\sdepth(I)$.

Let  $I=\cap_{i=1}^s P_i$  be a reduced intersection of monomial prime ideals of $S$ such that $P_i\not \subset \sum_{1=j\not =i}^s P_j$ for all $i\in [s]$. Then $$\sdepth_S I\geq \depth_S I=s+\dim S/\sum_{i=1}^sP_i,$$
as shows our Theorems \ref{-s} and \ref{s}. On the other hand, we show that if  $I$ is an
intersection of four monomial prime ideals then again Stanley's Conjecture holds (see Theorem \ref{main}).

 We introduce the so called the big size $t(I)$ of $I$ (usually bigger than the size of $I$ given in \cite{L}) and use it to find depth formulas. If $t(I)=1$ then $\depth I=2$ and the Stanley's Conjecture holds (see Corollary \ref{t1}). If $t(I)=2$ then we describe the possible values of $\depth I$ (see Lemmas \ref{good}, \ref{bad}) although we cannot show always that the Stanley's Conjecture holds. The obstruction is hinted by Example \ref{hin} and  Remark \ref{r1}.

\vskip 0.5 cm

\section{Big size one}
Let  $I=\cap_{i=1}^s P_i$, $s\geq 2$  be an irredundant intersection of monomial prime ideals of $S$.   We assume that $\sum_{i=1}^s P_i=m=(x_1,\ldots,x_n)$.
\begin{Definition} {\em Let $e$ be the minimal number such that there exists $e$ prime ideals among $(P_i)$ whose sum is $m$. After \cite{L}  the  {\em size} of $I$ is $e-1$. We call the {\em big size} of $I$ the minimal number $t=t(I)<s$  such that the sum of all possible $(t+1)$ prime ideals of $\{P_1,\ldots,P_s\}$
is $m$. We  set $t(m)=0$. Clearly the big size of $I$ is bigger or equal than the size of $I$.
If $a=\sum_{i=1}^s P_i\not=m$ then let $v$ be the minimal number $t<s$  such that the sum of all possible $(t+1)$ prime ideals of $\{P_1,\ldots,P_s\}$
is $a$. We call $v+\dim S/a$ the {\em big size} of $I$.}
\end{Definition}

We  need in our proofs  the following elementary lemma.
\begin{Lemma} \label{e} Let $J$, $E$, $F$ be some monomial ideals of $S$. Then there exists a canonical exact sequence
$$0\to S/(J\cap E\cap F)\to S/(J\cap E)\oplus S/(J\cap F)\to S/(J\cap (E+F))\to 0.$$
\end{Lemma}
\begin{proof}
Since the ideals are monomial we have $J\cap (E+F)=(J\cap E)+(J\cap F)$. The above exact sequence follows now from the well known exact sequence
$$0\to S/(E'\cap F')\to S/ E'\oplus S/ F'\to S/(E'+F')\to 0.$$
\end{proof}
\begin{Lemma}\label{1}  Suppose that there exists $1\leq c<s$ such that $P_i+P_j=m$ for each $c<j\leq s$ and $1\leq i\leq c$.
 Then $\depth_S I=2$. In particular, if the big size of $I$ is $1$ then $\depth_S I=2$.
\end{Lemma}
\begin{proof}
Using the following exact sequence (apply the above lemma for the case $J=S$, $E=\cap_{i=1}^{c} P_i$, $F=\cap_{j>c}^{s} P_j$)
$$0\to S/I\to S/\cap_{i=1}^{c} P_i\oplus S/\cap_{j>c}^{s} P_j\to S/\cap_{i=1}^{c} \cap_{j>c}^s(P_i+P_j)=S/m\to 0$$
we get $\depth S/I=1$ by Depth Lemma \cite[1.3.9]{Vi}, because $(\cap_{i=1}^{c} P_i)+( \cap_{j>c}^{s} P_j)=\cap_{i=1}^{c} \cap_{j>c}^s(P_i+P_j)=m$ by distributivity, the ideals being monomials.
\end{proof}
\begin{Remark}\label{app} {\em By \cite[Proposition 2]{L} $\depth_SS/I$ is always greater or equal than the size of $I$. So if the  size of $I$ is $1$, then necessarily $\depth_SI\geq 2$.  The equality follows when  the big size of $I$ is $1$. It is well known that $\depth_S S/I$ is less than or equal to $\dim S/P$, where $P$ is one of the assocated primes of $I$ ( see \cite[Proposition 1.2.13]{BH}).}
\end{Remark}
\begin{Example}{\em Let $n=5$, $s=4$, $P_1=(x_1,x_5)$, $P_2=(x_2,x_5)$, $P_3=(x_3,x_5)$, $P_4=(x_1,x_2,x_3,x_4)$. Since $P_1+P_2+P_3\not =m$ the big size of $I=\cap_{i=1}^4 P_i$ is $3$. The above lemma gives $\depth_SS/I=1$ because $P_i+P_4=m$ for all $1\leq i\leq 3$. Note that here the size of $I$  is $1$. In fact the above lemma gives examples when $\depth_SS/I=1$ and $t(I)\geq c$ for all positive integer $c$.}
\end{Example}
Next we  extend \cite[Proposition 2.3]{AP}. Let $r<n$ be a positive integer and $S'=K[x_{r+1},\ldots,x_n]$, $S''=K[x_{1},\ldots,x_r]$. We suppose that one prime ideal $P_i$ is generated by  some of the first $r$ variables. If $P_i=(x_1,\ldots,x_r)$ we say that $P_i$ is a {\em main prime}. For a subset $\tau\subset [s]$ we set
$$S_{\tau}=K[\{x_i: 1\leq i\leq r, x_i\not\in \Sigma_{j\in \tau} P_j \}]$$
 and let
${\mathcal F}$ be the set of all nonempty subsets  $\tau\subset [s]$ such that
$$L_{\tau}=(\cap_{i\in \tau} P_i)\cap S'\not =(0), \ \ J_{\tau}=(\cap_{i\in [s]\setminus \tau}\ P_i)\cap S_{\tau}\not =(0).$$
  For a $\tau\in {\mathcal F}$ consider the ideals
 $I_0=(I\cap K[x_1,\ldots,x_r])S$,  and
 $$I_{\tau}=J_{\tau}S_{\tau}[x_{r+1},\ldots,x_n]\cap L_{\tau}S_{\tau}[x_{r+1},\ldots,x_n]\subset S_{\tau}[x_{r+1},\ldots,x_n].$$
   Write $A_{\tau}=\sdepth_{S_{\tau}[x_{r+1},\ldots,x_n]}I_{\tau}$ which is at least $ \sdepth_{S_{\tau}}J_{\tau}+\sdepth_{S'}L_{\tau}$ by \cite[Theorem 4.1]{PQ}, \cite[Lemma 1.2]{AP}.   We also take $A_0=\sdepth_SI_0$ if $I_0\not=(0)$, otherwise take $A_0=n$.
\begin{Theorem}\label{sdepth} In the above setting
$\sdepth_SI\geq \min (\{A_0\}\cup \{A_{\tau}\}_{\tau\in {\mathcal F}}\}).$
\end{Theorem}
\begin{proof} (after \cite{AP}) First we show that
$$I=I_0\oplus (\oplus_{\tau\in {\mathcal F}}I_{\tau}),$$
where the direct sum is of linear $K$-spaces. Let $a\in I\setminus I_0$ be a monomial. We have $a=uv$, where $u\in S''$ and $v\in S'$. Set
$\rho =\{i\in [s]:u\not \in P_i\}$. Clearly, $\rho\not=\emptyset$ because $a\not \in I_0$. As $a\in I\subset P_i$,  we get $v\in P_i$ for all  $i\in \rho$, and $v\in L_{\rho}$. On the other hand, by definition of $\rho$ we have $u\in J_{\rho}$. Hence $\rho\in {\mathcal F}$ and $a\in I_{\rho}$. The sum is direct because for any $a\in I\setminus I_0$ there exists just one $\rho=\{i\in [s]:u\not \in P_i\}\in {\mathcal F}$ such that $a\in I_{\rho}$. Note that the monomials of $I\setminus I_0$ are disjoint union of monomials of $I_{\tau}$, $\tau\in {\mathcal F}$.

Now choose ``good" Stanley decompositions ${\mathcal D}_0$, ${\mathcal D}_{\tau}$ for $I_0$, respectively $I_{\tau}$ such that
$\sdepth_S{\mathcal D}_0=\sdepth_SI_0$, $\sdepth_{S_{\tau}[x_{r+1},\ldots, x_n]}{\mathcal D}_{\tau}=\sdepth_{S_{\tau}[x_{r+1},\ldots, x_n]}I_{\tau}.$
They will induce a Stanley decomposition ${\mathcal D}$ of $I$ such that

$\sdepth_SI\geq \sdepth_S{\mathcal D}=\min ( \{\sdepth_SI_0\}\cup\{\sdepth_{S_{\tau}[x_{r+1},\ldots, x_n]}I_{\tau}\}_{\tau\in {\mathcal F}}\}).$\end{proof}

\begin{Corollary}\label{t1} If the big size of $I$ is $1$ then  $\sdepth_S I\geq 2$, that is Stanley's Conjecture holds for $I$.
\end{Corollary}
\begin{proof} It is easy to see that the corollary holds for $n\leq 2$. If $n\geq 3$ then $\sdepth_SI\geq 2=\depth I$ by \cite[Theorem 3.4]{FH}, which is enough as shows our Lemma \ref{1}. For the sake of the completeness we give below another proof applying the above proposition.

Use induction on $s\geq 1$, the case $s=1$ being easy. We may assume that $P_1=(x_1,\ldots,x_r)$ for some $r<n$.
By Theorem \ref{sdepth} we have $$\sdepth_SI\geq \min (\{A_0\}\cup\{A_{\tau_i}\}_{\tau_i\in  {\mathcal F}}\}),$$
where $\tau_i=\{i\}$ for some $1<i\leq s$. Indeed, we have ${\mathcal F}\subset\{\tau_i\}_{1<i\leq s}$ because $P_j+P_i=m$ for all $j\not =i$. The inclusion is in fact an equality. Indeed, if
$$P_j\cap K[\{x_e: 1\leq e\leq r, x_e\not \in P_i\}]=(0)$$
 for some $1<j\not =i$ then $P_j\cap S''\subset P_i$ and so $P_1\subset P_j$ since $P_j+P_i=m$ (contradiction).
 If $I_0\not =(0)$ then
$$A_0= \sdepth_{ S''} (I\cap S'')+n-r\geq 1+\dim S/P_1\geq 1+\depth S/I=\depth I.$$
On the other hand, we have
$$A_{\tau_i}\geq\sdepth_{S_{\tau_i}}(\cap_{j\not =i}  P_j\cap S_{\tau_i})+\sdepth_{S'}(P_i\cap S')\geq $$
$$ \depth_{S_{\tau_i}}(\cap_{j\not =i}  P_j\cap S_{\tau_i})+\depth_{S'}((x_{r+1},\ldots ,x_n)S')\geq 2$$
by induction hypothesis and
because $P_j+P_i=m$ for all $j\not =i$. As $\depth I=2$ by Lemma \ref{1} we are done.
\end{proof}
\vskip 0.5 cm

\section{Some results of general big size}
 Let $I=\cap_{i=1}^s P_i$, $s\geq 2$  be an irredundand  intersection of monomial prime ideals of $S$.
\begin{Lemma}\label{ex} If $P_1\not \subset \sum_{i=2}^{s}P_i$ then
$$\depth I=\min (\depth (\cap_{i=2}^{s}P_i),1+\depth (\cap_{i=2}^{s}(P_i+P_1))).$$
\end{Lemma}
\begin{proof}
By Lemma \ref{e} we have the following exact sequence
$$0\to S/I\to S/(\cap_{i=2}^{s}P_i)\oplus S/P_1\to S/(\cap_{i=2}^{s}(P_i+P_1))\to 0$$
where  $\depth S/I\leq \depth S/P_1$ by \ref{app}. Choosing a variable  $x_i\in P_1\setminus \Sigma_{i=2}^{s}P_i$ we see that $I:x_i=
\cap_{i=2}^{s}P_i$. So
$$\depth S/I\leq \depth S/(I:x_i)=\depth S/(\cap_{i=2}^{s}P_i)$$
by \cite[Corollary 1.3]{R}. It follows that
$$\depth S/I=\min (\depth S/(\cap_{i=2}^{s}P_i),1+\depth S/(\cap_{i=2}^{s}(P_i+P_1)))$$
from Depth Lemma (see \cite[Lemma 1.3.9]{Vi}), because \\
$\depth_SS/P_1\geq 1+\depth_SS/(\cap_{i=2}^{s}(P_i+P_1)).$
\end{proof}

The next theorem uses an easy lemma of Ishaq \cite[Lemma 3.1]{Is1}.
\begin{Lemma}(Ishaq) \label{is}
Let $J\subset S[y]$ be a monomial ideal, $y$ being a new variable. Then $\sdepth_S(J\cap S)\geq \sdepth_{S[y]} J-1$.
\end{Lemma}
The following theorem extends \cite[Theorem 1.4]{AP}.
\begin{Theorem} \label{-s}  Let  $I=\cap_{i=1}^s P_i$  be a reduced intersection of monomial prime ideals of $S$. Assume that $P_i\not \subset \sum_{1=j\not =i}^s P_j$ for all $i\in [s]$. Then $$\depth_S I=s+\dim S/\sum_{i=1}^sP_i.$$
\end{Theorem}
\begin{proof} By \cite[Lemma 3.6]{HVZ} it is enough to consider the case when $\sum_{j=1}^s P_j=m$. Apply induction on $s$. If $s=1$  the result is trivial because $\depth_Sm=1$. Suppose that $s>1$. We may assume that
$P_1=(x_{1},\ldots,x_r)$ for some $r<n$ and set $S''=K[x_1,\ldots,x_r]$, $S'=K[x_{r+1},\ldots,x_n]$.
By Lemma \ref{ex} we get
$$\depth_SI=\min (\depth_S (\cap_{i>1}^{s}P_i),1+\depth_S (\cap_{i>1}^{s}(P_i+P_1))).$$
Note that $P_i\not \subset \Sigma_{1<j\not =i}^{s} P_j$ for all $1<i\leq s$ because, otherwise, we contradict the hypothesis.   Then  the induction hypothesis gives
$$\depth_{S}(\cap_{j>1}^{s} P_j)= s-1+     \dim S/(\Sigma_{i>1}^s P_i)\geq s.$$
As $\cap_{i>1}^{s}(P_i+P_1)$ satisfies also our assumption,  the induction hypothesis gives

\noindent $\depth_S(\cap_{i>1}^{s}(P_i+P_1))=s-1$. Hence $\depth_S I=s$.
\end{proof}
\begin{Theorem} \label{s}  Let  $I=\cap_{i=1}^s P_i$  be a reduced intersection of monomial prime ideals of $S$. Assume that $P_i\not \subset \sum_{1=j\not =i}^s P_j$ for all $i\in [s]$. Then $$\sdepth_S I\geq \depth_S I,$$
that is  Stanley's Conjecture holds for $I$.
\end{Theorem}
\begin{proof} As in the above theorem we may consider only the case $\sum_{j=1}^s P_j=m$. Apply induction on $s$.
We apply Theorem \ref{sdepth} for ${\mathcal F}$ containing  as usual some $\tau\subset [s]$. Note that $P_1\cap S'=(0)$ since $P_1$ is generated in the first $r$ variables. Thus $\tau\in {\mathcal F}$ cannot contain $1$ by the construction of $\mathcal F$.
 We get $\sdepth_S I\geq \min (\{A_0\}\cup\{A_{\tau}\}_{\tau\in {\mathcal F}}\})$ for
$A_0=\sdepth(I\cap S'')S$ if $I\cap S''\not=0$ or $A_0=n$ otherwise, and
$$A_{\tau}\geq\sdepth_{S_{\tau}} ((\cap_{i\not \in\tau}P_i)\cap S_{\tau})+\sdepth_{S'}(\cap_{i\in \tau}P_i\cap S'),$$
where $S_{\tau}=K[\{x_i: 1\leq i\leq r, x_i\not\in \Sigma_{j\in \tau} P_j \}]$.
Note that  $\cap_{j\in \tau}P_j\cap S'$ satisfies  our assumption  because if $P_k\cap S'\subset \Sigma_{j\in \tau, j\not =k} (P_k\cap S')$ and  we get $P_k\subset  \Sigma_{1=j\not =k}^s P_j$ which is false. Thus by induction hypothesis we have
$$\sdepth_{S'}(\cap_{i\in \tau}P_i\cap S')\geq
\depth_{S'}(\cap_{i\in \tau}P_i\cap S')=$$
$$|\tau|+\dim S'/(\cap_{i\in \tau}P_i\cap S')=|\tau|+\dim S/(P_1+\Sigma_{i\in \tau} P_i),$$
using Theorem \ref{-s}.
Let ${\tilde S}_{\tau}=S_{\tau}[\{x_j:j>r, x_j\not \in  \Sigma_{i\in \tau} P_i)\}].$ Note that $(\cap_{i\not \in\tau}P_i)\cap {\tilde S}_{\tau}$ satisfies our hypothesis even though
 $(\cap_{i\not \in\tau}P_i)\cap S_{\tau}$ may  not. Indeed, if $P_i\cap {\tilde S}_{\tau}  \subset \Sigma_{j\not\in \tau, j\not =i} P_j$
 for some $i\not \in \tau$ then $P_i \subset \Sigma_{1=j\not =i}^s P_j$ which is false. By Lemma \ref{is} we have
 $$\sdepth_{S_{\tau}} ((\cap_{i\not \in\tau}P_i)\cap S_{\tau})\geq \sdepth_{{\tilde S}_{\tau}} ((\cap_{i\not \in\tau}P_i)\cap {\tilde S}_{\tau})-|\{i>r:x_i\not \in \Sigma_{j\in \tau} P_j\}|\geq $$
 $$s-|\tau|-\dim S/(P_1+\Sigma_{i\in \tau} P_i),$$
 using the induction hypothesis.
 Thus $A_{\tau}\geq s=\depth_SI$ by the above theorem.
 Finally note that if $I\cap S''\not =(0)$ then $$A_0=\sdepth_{S''} (I\cap S'')+n-r\geq 1+\dim S/P_1\geq 1+\depth_S S/I=\depth_SI$$ using \cite[Lemma 3.6]{HVZ}.
\end{proof}

\vskip 0.5 cm
\section{Depth on big size two}
 Let $I=\cap_{i=1}^s P_i$, $s\geq 3$  be a reduced intersection of monomial prime ideals of $S$.    Assume that $\sum_{i=1}^s P_i=m$
and the big size of $I$ is two. We may suppose that $P_1+P_2=(x_1,\ldots,x_r)$ for some $r<n$. We set $$q=\min (\dim S/(P_i+P_j): j\not =i,P_i+P_j\not =m).$$ Thus $q\leq n-r.$ Set $S''=K[x_1,\ldots,x_r]$, $S'=K[x_{r+1},\ldots,x_n]$.
\begin{Lemma}\label{eas} $\depth_S S/I\leq 1+q$.
\end{Lemma}
\begin{proof}  Note that for $i>2$ we have $P_i\not \subset P_1+P_2$. This is because, otherwise, $P_1+P_2=P_1+P_2+P_i=m$
by the condition  $t(I)=2$, which gives a contradiction. Then we may find a monomial $u\in \cap_{i>2}^s P_i\setminus (P_1+P_2)$ and we have $(I:u)=P_1\cap P_2$. Thus $$\depth_SS/I
\leq  \depth_S S/(I:u)=\depth_S S/(P_1\cap P_2)=1+\dim S/(P_1+P_2)$$ by \cite[Corollary 1.3]{R}, the last equality being a consequence of  Depth Lemma applied to the exact sequence
$$0\to S/(P_1\cap P_2)\to S/P_1\oplus S/P_2\to S/(P_1+P_2)\to 0.$$
 In this way we see that
$$\depth_SS/I\leq 1+\min (\dim S/(P_i+P_j): j\not =i,P_i+P_j\not =m).$$
\end{proof}

\begin{Lemma} \label{good} If $P_k+P_e=m$ for all distinct $k,e>2$,  then the following statements hold:
\begin{enumerate}
\item{} $\depth_S S/I\in \{1,2,1+q\}$,
\item{} $\depth_S S/I=1$ if and only if there exists $j>2$ such that $P_1+P_j=m=P_2+P_j$,
\item{}  $\depth_S S/I>2$ if and only if $q>1$ and each $j>2$ satisfies either $$P_1+P_j\not=m =P_2+P_j,\ \mbox{or}$$
$$P_2+P_j\not=m =P_1+P_j,$$
\item{}  $\depth_S S/I=2$ if and only if  both the following conditions hold:
\begin{enumerate}
\item{}  each $j>2$ satisfies either $P_1+P_j\not=m$ or
$P_2+P_j\not=m,$
\item{} $q=1$  or there exists an index $k>2$ such that  $$P_1+P_k\not=m \not=P_2+P_k.$$
\end{enumerate}

\end{enumerate}
\end{Lemma}

\begin{proof}
Apply induction on $s+n$, $s\geq 3$. If $s=3$ then we may apply \cite[Proposition 2.1, Theorem 2.6]{AP}. Suppose that $s>3$.
By Lemma \ref{e} applied for $J=P_1\cap P_2$, $E=P_3$, $F=P_4+\ldots +P_s$ we have the following exact sequence
$$0\to S/I\to S/(P_1\cap P_2\cap P_3)\oplus S/(P_1 \cap P_2\cap P_4\cap \ldots\cap P_s)\to S/(P_1\cap P_2)\to 0$$
because $P_3+P_k=m$ for all $k>3$. Using  \cite[Corollary 1.3]{R} as in the proof of Lemma \ref{eas}, any module from the above exact sequence has depth $\leq
\depth_S S/(P_1\cap P_2)$. Thus
$$\depth_SS/I=\min (\depth_S  S/(P_1\cap P_2\cap P_3),\depth_S S/(P_1 \cap P_2\cap P_4\cap \ldots\cap P_s) )$$
by Depth Lemma \cite[Lemma 1.3.9]{Vi}. Using the induction hypothesis, we get
$$\depth_S  S/(P_1\cap P_2\cap P_3), \ \ \depth_S S/(P_1 \cap P_2\cap P_4\cap \ldots\cap P_s)\in \{1,2,1+q\}$$
 because any three prime ideals of $(P_i)$ have the sum $m$.
Hence (1) holds. Note that $\depth_SS/I=1$ if and only if either $\depth_S  S/(P_1\cap P_2\cap P_3)=1$, or \\
$\depth_S S/(P_1 \cap P_2\cap P_4\cap \ldots\cap P_s)=1$
and (2) holds because of the induction hypothesis (see also Lemma \ref{1}). Similarly, (3), (4)  holds by induction hypothesis relying in fact on the case $s=3$ stated in \cite{AP}.
\end{proof}
\begin{Lemma} \label{good1} If $P_k+P_e=m$ for all distinct $k,e>2$,  then
 $\sdepth_S I\geq \depth_S I$.
\end{Lemma}
\begin{proof}
We apply Theorem \ref{sdepth} to ${\mathcal F}$ containing some $\tau_i=\{ i\}$, $2<i\leq s$
(note that $P_i+P_j=m$ for all $2<i<j\leq s$ and so $\mathcal F$ does not contain $\tau=\{i,j\}$).
We get $\sdepth_S I\geq \min (\{A_0\}\cup\{A_{\tau_i}\}_{\tau_i\in {\mathcal F}}\})$ for
$A_0=\sdepth_S(I\cap S'')S$ if $I\cap S''\not=0$ or $A_0=n$ otherwise, and
$$A_{\tau_i}\geq \sdepth_{S_{\tau_i}} ((\cap_{j=1, j\not =i}^{s}P_j)\cap S_{\tau_i})+\sdepth_{S'}(P_i\cap S'),$$
where $S_{\tau_i}=K[\{x_j:x_j\in S'',x_j\not \in P_i\}]$. Note that the big size of $J_i=(\cap_{j=1, j\not =i}^{s}P_j)\cap S_{\tau_i}$ is 1 or zero,
because if $(P_k+P_e)\cap S_{\tau_i}$ is not the maximal ideal of $S_{\tau_i}$ for some two different  $k,e$ which  are not $i$,  then  $P_k+P_e+P_i\not =m$ contradicting $t(I)=2$.
By Corollary \ref{t1} we get $$\sdepth_{S_{\tau_i}}J_i\geq \depth_{S_{\tau_i}}J_i=1+\depth_{S_{\tau_i}} S_{\tau_i}/J_i=1+\depth_S S/(J_iS+P_i).$$
Then $A_{\tau_i}\geq 2+\depth_S S/(J_iS+P_i)$. By our hypothesis $$J_iS+P_i=((P_1\cap S_{\tau_i})S\cap (P_2\cap S_{\tau_i})S)+P_i.$$ But $(P_k\cap  S_{\tau_i})S+P_i=P_k+P_i$ for $k=1,2$ and so
$J_iS+P_i=(P_1+P_i)\cap (P_2+P_i)$. If $P_1+P_i=m\not =P_2+P_i$ then $\depth_S S/(J_iS+P_i)=\dim S/(P_2+P_i)\geq q$. Hence $A_{\tau_i}\geq \depth_S I $ using (1) of the above lemma. If $P_1+P_i\not =m\not =P_2+P_i$ then we get $A_{\tau_i}\geq 3=\depth_S I$ using (4) of the above lemma.

 Suppose that $I\cap S''\not=0$. When $t(I\cap S'')= 1$ we have $\sdepth_{S''}(I\cap S'')\geq 2$ by Corollary \ref{t1} and so $A_0\geq 2+n-r\geq 2+q\geq \depth_SI$. When  $t(I\cap S'')= 2$, since less variables are involved, we can use the induction hypothesis and we have
$$A_0\geq \depth_S(I\cap S'')S= n-r+\depth_{S''}(I\cap S'')\geq q+2\geq  \depth_SI.$$
Note that in this case $I\cap S''$ cannot be the homogeneous maximal ideal in $S''$.
\end{proof}

Next we will consider another case when $t(I)=2$, but with the following property:

($*$) whenever there exist $i\not =j$ in $[s]$ such that $P_i+P_j\not =m$  there exist also $k\not =e$ in $[s]\setminus\{i,j\}$
such that $P_k+P_e\not =m$.

\noindent This is exactly the complementary case to the one solved by the above lemma. As before we may suppose that $P_1+P_2\not =m$. Now by ($*$) we may  suppose $P_s+P_{s-1}\not =m$.

\begin{Lemma}\label{bad} If $t(I)=2$ and $I$ satisfies ($*$) then
\begin{enumerate}
\item{} $\depth_S S/I\in \{1,2,1+q\}$.
\item{}  $\depth_S S/I=1$ if and only if  after a renumbering of $(P_i)$ there exists $1\leq c<s$ such that $P_i+P_j=m$ for each $c<j\leq s$ and $1\leq i\leq c$.
\end{enumerate}
\end{Lemma}
\begin{proof} We  use induction on $s\geq 3$, with the case $s=3$ having been covered  in \cite[Proposition 2.1, Theorem 2.6]{AP}. Now we assume $s>3$ and set $J= P_1\cap \ldots \cap P_{s-2}$. Since $t(I)=2$, $P_i+P_{s-1}+P_s=m$ for all $i<s-1$. Note that there exist no $i<s-1$ such that $P_i\subset P_{s-1}+P_s$ because otherwise $P_{s-1}+P_s=P_i+P_{s-1}+P_s=m$, which is false.
Thus, in the exact sequence (apply Lemma \ref{e})
$$0\to S/I\to S/(J \cap P_{s-1})\oplus S/(J\cap P_s)\to S/(J\cap (P_{s-1}+P_s))\to 0$$
we have  $\depth_S  S/(J\cap (P_{s-1}+P_s))=1$ by Lemma \ref{1}.
If
$$(+)\ \ \depth_S(S/(J \cap P_{s-1})\oplus S/(J\cap P_s))>1$$
 then
$\depth_S S/I=2$. Otherwise, we may suppose that $\depth_S(S/(J \cap P_{s-1}))=1$, where we  apply part (2) of Lemma \ref{good}.
Thus, after a renumbering of $(P_i)$, there exists $1\leq k<s-1$ such that $P_i+P_j=m$ for each $k<j\leq s-1$ and $1\leq i\leq k$.
In the following exact sequence (again apply Lemma \ref{e} for $J=P_s$, $E=P_1\cap\ldots\cap P_k$, $F=P_{k+1}\cap \ldots\cap P_{s-1}$)
 $$0\to S/I\to S/( P_1\cap \ldots \cap P_{k} \cap P_{s})\oplus S/(P_{k+1}\cap \ldots\cap P_s)\to S/P_s\to 0$$
 all the modules have depth $\leq \depth_S S/P_s$ by \ref{app}. It follows
 $$ \depth_S S/I=\min (\depth_SS/( P_1\cap \ldots \cap P_{k} \cap P_{s}), \depth_S S/(P_{k+1}\cap \ldots\cap P_s))$$
and applying Lemma \ref{good}  we get (1).

In (2) the sufficiency follows from Lemma \ref{1}.
If $\depth_S S/I=1$  we will get,  say,
 $\depth_SS/( P_1\cap \ldots \cap P_{k}  \cap P_{s})=1$. Now
use Lemma \ref{good} and our induction hypothesis. After a renumbering of $(P_i)_{i<k}$ there exists $1\leq c\leq k$ such that $P_i+P_j=m$ for each
$1\leq i\leq c$ and $c<j\leq k$ or $j=s$. Thus, using our assumptions on $k$ we get $P_i+P_j=m$ for each $c<j\leq s$ and $1\leq i\leq c$.
\end{proof}
\vskip 0.5 cm
\section{Intersections of four prime ideals}
Let $I=\cap_{i=1}^4 P_i$   be an  irredundant intersection of monomial prime ideals of $S$.    Assume that $\sum_{i=1}^4 P_i=m$
and the big size of $I$ is two. Thus we may further assume  $P_1+P_2\not =m$ and $P_1=(x_1,\ldots,x_r)$, $r<n$.
  Set $$q=\min (\dim S/(P_i+P_j): j\not =i,P_i+P_j\not =m),$$  $S''=K[x_1,\ldots,x_r]$, $S'=K[x_{r+1},\ldots,x_n].$
\begin{Proposition} \label{2,4}  In the above setting  $\sdepth_SI\geq \depth_SI$.
\end{Proposition}
\begin{proof} Using Lemma \ref{good1} we may suppose that $I$ satisfies ($*$) and  $P_3+P_4\not =m$.  If
$\depth_S(S/(P_1\cap P_2 \cap P_3)\oplus S/(P_1\cap P_2\cap P_4))>1$, the proof of Lemma \ref{bad} (see $(+)$) shows that
 $\depth_SS/I=2$. Otherwise, we may assume that $\depth_SS/(P_1\cap P_2 \cap P_3)=1$. It follows from \cite[Proposition 2.1]{AP} $P_1+P_3=P_2+P_3=m$,  since $P_1+P_2\not =m$.
Then ($*$)  implies that $P_1+P_4=P_2+P_4=m$ and we have  $\depth_SS/I=1$ by Lemma \ref{1}. Thus $\depth_SI\leq 3$ if $I$ satisfies ($*$) even $\depth_SI= 2$ if  $P_1+P_3=P_1+P_4=P_2+P_3=P_2+P_4=m$.

Apply Theorem \ref{sdepth} for the main prime $P_1$ and
${\mathcal F}$ containing only possible $\tau_i=\{ i\}$, $i=2,3,4$,   $\tau_{ij}=\{i,j\}$ for some $1<i<j\leq 4$. We get $\sdepth I\geq \min (\{A_0\}\cup\{A_{\tau}\}_{\tau\in {\mathcal F}}\})$. As usual,
$A_0=\sdepth(I\cap S'')S$ if $I\cap S''\not=0$ or $A_0=n$ otherwise. We have
$$A_{\tau_i}\geq\sdepth_{S_{\tau_i}} ((\cap_{j=2, j\not =i}^{4}P_j)\cap S_{\tau_i})+\sdepth_{S'}(P_i\cap S'),$$
for $i=2,3,4$ and
$$A_{\tau_{ij}}\geq\sdepth_{S_{\tau_{ij}}} (P_k\cap S_{\tau_{ij}})+\sdepth_{S'}(P_i\cap P_j\cap S'),$$
where $1<i<j\leq 4$, $k=[4]\setminus \{1,i,j\}$. Here we set \\
 $S_{\tau_{ij}}=K[x_j:x_j\in S'',x_j\not \in P_i+P_j]$ and $S_{\tau_i}=K[x_j:x_j\in S'',x_j\not \in P_i]$.
 As in Lemma \ref{good} we have $A_0\geq \depth I$. It is enough to show that $A_{\tau_i}, A_{\tau_{ij}}\geq 3$ except in the case $P_1+P_3=P_1+P_4=P_2+P_3=P_2+P_4=m$ when it is enough to show that $A_{\tau_{34}}\geq 2$. Note that $A_{\tau_2}\geq  3$ because $\sdepth_{S'}(P_2\cap S')\geq 1+\lceil\frac{\height (P_2\cap S')}{2}\rceil$.

\vskip 0.3 cm
{\bf Part $A_{\tau_i}\geq 3$, $i>2$}

 We study for example $A_{\tau_4}$. Using  \cite[Lemma 4.3]{PQ} we have
$$A_{\tau_4}\geq
\sum_{j=2}^{3}\lceil\frac{\dim S''/((P_j+P_4)\cap S'')}{2}\rceil +1\geq 3,$$
if $(P_2+P_4)\cap S''$ and $(P_3+P_4)\cap S'' $ do not contain each other, where $\lceil a\rceil$, $a\in {\bf Q}$ denotes the smallest integer not less than $ a$. Otherwise, if $P_2\cap S''\subset P_3+P_4$ then  $P_2\cap S'\not \subset P_4$ since $P_2+P_3+P_4=m$ and $P_3+P_4\not =m$. Thus $P_4\cap S'$ is not the maximal ideal of $S'$ and so
$\sdepth_{S'}(P_4\cap S')\geq 1+\lceil\frac{\height (P_4\cap S')}{2}\rceil$ by \cite{Bi}. Then
$$A_{\tau_4}\geq
\sdepth_{S_{\tau_4}} (P_2\cap S_{\tau_4})+\sdepth_{S'}(P_4\cap S')\geq 2+\lceil\frac{\height (P_4\cap S')}{2}\rceil .$$
If $P_3\cap S''\subset P_2+P_4$ and  $P_3\cap S'\not \subset P_4$ we proceed as above. If $P_3\cap S' \subset P_4$ then we get $P_2+P_4 =m$ because $P_2+P_3+P_4=m$.
By ($*$) we get also $P_1+P_3=m$. It follows $P_3\cap S_{\tau_4}$ is not maximal in $S_{\tau_4}$ because $P_3+P_4\not =m$ and so
$$A_{\tau_4}\geq
\sdepth_{S_{\tau_4}} (P_3\cap S_{\tau_4})+\sdepth_{S'}(P_4\cap S')\geq 2+\lceil\frac{\height (P_3\cap S_{\tau_4})}{2}\rceil.$$
\vskip 0.3 cm
{\bf Part  $A_{\tau_{ij}}\geq 3$}

Next, by  \cite[Lemma 4.3]{PQ} $$A_{\tau_{34}}\geq \sdepth_{S_{\tau_{34}}} (P_2\cap S_{\tau_{34}})+\sdepth_{S'}(P_3\cap P_4\cap S')\geq$$
$$\lceil\frac{\height (P_2\cap S_{\tau_{34}})}{2}\rceil+\lceil\frac{\dim S'/(P_3\cap S')}{2}\rceil +\lceil\frac{\dim S'/(P_4\cap S')}{2}\rceil\geq 3$$
if $P_3\cap S'$ and $P_4\cap S' $  do not contain each other (note that $P_2+P_3+P_4=m$). Otherwise, if for example $P_3\cap S'\subset P_4$ we get $P_1+P_4=m$ because $P_1+P_3+P_4=m$, and so $P_2+P_3=m$ by ($*$). If $P_1+P_3\not =m$ then $P_3\cap S'$ is not the maximal ideal of $S'$. It follows that
$\sdepth_{S'}(P_3\cap S')\geq 1+ \lceil\frac{\height (P_3\cap S')}{2}\rceil$. Thus, $A_{34}\geq 3$.
On the other hand,
if $P_1+P_3=m$, then $P_2+P_4=m$ by ($*$) and so $A_{34}\geq 2= \depth_SI$ as we know already.
Similarly, if $\tau_{23}\in {\mathcal F}$ we get $A_{\tau_{23}}\geq 3$ if $P_2\cap S'\not \subset P_3\cap S'$,
otherwise we see that $P_2\cap S'$ is not the maximal ideal in $S'$ and so $$A_{\tau_{23}}\geq 2+\lceil\frac{\height (P_2\cap S')}{2}\rceil \geq 3.$$
\end{proof}

\begin{Theorem} \label{main} Let  $I=\cap_{i=1}^4 P_i$  be a reduced  intersection of  four monomial prime ideals of $S$. Then  Stanley's Conjecture holds for $I$.
\end{Theorem}
\begin{proof} By \cite[Lemma 3.6]{HVZ} it is enough to consider the case when $\sum_{j=1}^4 P_j=m$. If $t(I)\leq 2$ then the result follows by Corollary \ref{t1} and Proposition \ref{2,4}.
Otherwise, there exists $i\in [s]$ such that $P_i\not \subset \sum_{1=j\not =i}^4 P_j$, let us say  $P_4\not \subset \sum_{1=1}^3 P_j$.
Apply induction on $n$, the case $n\leq 5$ being done in \cite{P}. We assume that $\sum_{1=1}^3 P_j=(x_1,\ldots,x_r)$ for some $r<n$. Apply Theorem \ref{sdepth} as before with $\mathcal F$ containing just $\tau=\{4\}$. We have
$$A_{\tau}\geq \sdepth_{S_{\tau}} ((\cap_{j=1}^{3}P_j)\cap S_{\tau})+\sdepth_{S'}(P_4\cap S')\geq  \depth_{S_{\tau}} ((\cap_{j=1}^{3}P_j)\cap S_{\tau})+1$$
by \cite{AP} and so
$$A_{\tau}\geq \depth_{S_{\tau}} S_{\tau}/((\cap_{j=1}^{3}P_j)\cap S_{\tau})+2=2+\depth_{S} S/((\cap_{j=1}^{3}(P_j+P_4))=$$
$$1+\depth_{S} ((\cap_{j=1}^{3}(P_j+P_4))\geq \depth_SI$$
by Lemma \ref{ex}. Suppose $I\cap S''\not =0$. Then $A_0\geq n-r+\sdepth_{S''}(I\cap S'')$ by \cite[Lemma 3.6]{HVZ}. If $t(I\cap S'')\leq 2$ we get $\sdepth_{S''}(I\cap S'')\geq \depth_{S''}(I\cap S'')$ as above. Otherwise there exists $i\in [4]$ such that $(P_i\cap S'')\not \subset \sum_{1=j\not =i}^4 (P_j\cap S')$ and we get a similar estimate using the induction hypothesis (less variables). Thus $A_0\geq n-r+\depth_{S''}(I\cap S'')\geq\depth_SI$ by \cite[Proposition 1.2]{R}.
\end{proof}

\begin{Example}\label{hin}{\em Let $n=10$, $P_1=(x_1,\ldots,x_7)$, $P_2=(x_3,\ldots,x_8)$,

\noindent $P_3=(x_1,\ldots,x_4,x_8,\ldots,x_{10})$, $P_4=(x_1,x_2,x_5,x_8,x_9,x_{10})$,
$P_5=(x_5,\ldots,x_{10})$. We have $P_1+P_3=P_2+P_3=P_1+P_4=P_2+P_4=P_3+P_5=P_1+P_5=m$, $P_2+P_5=m\setminus\{x_1,x_2\}$, $P_3+P_4=m\setminus\{x_6,x_7\}$,
$P_4+P_5=m\setminus\{x_3,x_4\}$, $P_1+P_2=m\setminus\{x_9,x_{10}\}$. We have $t(I)=2$. Applying the proof of  Lemma \ref{bad} (see there the last exact sequence), we get $$\depth_SS/I=
\min\{\depth_SS/(P_1\cap P_2), \depth_SS/(P_2\cap \ldots\cap P_5)\}.$$
We have $\depth_SS/(P_1\cap P_2)=3$ and for $a:=\depth_SS/(P_2\cap \ldots\cap P_5)$ we apply (3) of Lemma \ref{good}, with $P_4+P_5\not =m$ and $P_2+P_3=m$.
As for $j=2$ we have $P_2+P_4=m\not =P_2+P_5$ and for $j=3$ we have $P_3+P_4\not =m=P_3+P_5$ it
follows that $a=1+\dim S/(P_4+P_5)=3$ and so $\depth_S I=4$.

Applying Theorem \ref{sdepth} to $P_1$ as main prime we see that $A^{(1)}_{3,4}\geq 3$, where $A^{(1)}_{3,4}$ denotes $A_{\tau}$ when $P_1$ is the main prime for $\tau=\{3,4\}$. Indeed, $$A^{(1)}_{3,4}\geq
\sdepth_{K[x_6,x_7]} (x_6,x_7)K[x_6,x_7]+\sdepth_{K[x_8,x_9,x_{10}]}(x_8,x_9,x_{10})K[x_8,x_9,x_{10}]=3.$$
Similarly choosing $P_2$ as a main prime we get $A^{(2)}_{3,4}\geq 3$
 (now the usual $r$-variables are the variables generating $P_2$, namely $x_3,\ldots,x_8$)
 and taking $P_3$,$P_4$    as  main primes we get $A^{(3)}_{2,5}\geq 3$, respectively $A^{(4)}_{2,5}\geq 3$.  Thus from these we cannot conclude that $\sdepth_S I\geq \depth_S I$. Fortunately, choosing $P_5$ as a main prime you can see that all $A_{\tau}\geq 4$, which is enough (notice that $\{2\}\not \in{\mathcal F}^{(5)}$). Note that $\dim S/P_5=4$ is maximum possible among $\dim S/P_i$, but we have also $\dim S/P_2=\dim S/P_4=4$.
}
\end{Example}
\begin{Remark}\label{r1}{\em
The above example shows that it is not clear how one can use the special Stanley decompositions from  \cite[Proposition 2.3]{AP}  (see here Theorem \ref{sdepth}) in general. It is not clear that we may
find always a "good" main prime $P_i$. If it really exists then it is not clear how we could  pick it, the maximum dimension of $S/P_i$ seems to be not enough.
 On the other hand, if we apply Theorem \ref{sdepth} for $r=8$, that is to the case $P_1+P_2=(x_1,\ldots,x_8)$,  then
$$A_5^{(12)}\geq \sdepth((x_3,x_4)\cap (x_1,x_2)\cap K[x_1,\ldots,x_4])+\sdepth ((x_9,x_{10})\cap K[x_9,x_{10}])=4,$$
$$\depth((x_3,x_4)\cap (x_1,x_2)\cap K[x_1,\ldots,x_4])+\depth ((x_9,x_{10})\cap K[x_9,,x_{10}])=3<\depth_SI.$$
Thus, we cannot hope to prove  the Stanley's Conjecture, in general, by induction on $n$, using these special Stanley decompositions.}
\end{Remark}

\end{document}